\documentclass[a4paper,12pt]{amsart} 

\usepackage{graphics,epsfig}
\usepackage{xcolor}
\definecolor{darkgreen}{rgb}{0,0.41,0.11}

\usepackage{amssymb} 
\usepackage{amsthm} 
\usepackage{mathtools}
\usepackage{amsmath}

\usepackage{float}
\usepackage{enumerate}
\usepackage[normalem]{ulem}
\usepackage{romannum}
\usepackage{caption}
\usepackage{subcaption}
\usepackage{url}
\usepackage[pagewise,displaymath,mathlines]{lineno} 

\newcommand{\R}{\mathbb{R}}

\newcommand{\F}{\mathcal{F}}
\renewcommand{\H}{\mathbb{H}}

\newcommand{\diam}[1]{\text{diam}{#1}}

\DeclareMathOperator\arctanh{arctanh}
\renewcommand{\vec}{\mathbf}

\newtheorem{definition}{Definition}[section]

\newtheorem{prop}[definition]{Proposition}
\newtheorem{lem}[definition]{Lemma}
\newtheorem{coro}[definition]{Corollary}
\newtheorem{teo}{Theorem}

\theoremstyle{definition}
\newtheorem{exa}[definition]{Example}
\newtheorem{rem}[definition]{Remark}

\begin{document}

	
	\bigskip

	\title[]{Some results on evolutoids of convex curves in $2$-dimensional space forms}
		
	
	\author[]{A. C. Junior}
	\address[Ady]{Departamento de Matem\'atica, Universidade Federal de Vi\c cosa, Brazil.
	}
	\email{ady.cambraia@ufv.br}
	\author[]{A. G. Chimenton}
	\address[Alessandro]{Departamento de Matem\'atica, Universidade Federal Fluminense campus Volta Redonda, Brazil.}
	\email{alessandrogaio@id.uff.br}
	\author[]{M. A. C. Fernandes}
	\address[Marco]{%
		Departamento de Matem\'atica, Universidade Federal de Vi\c cosa, Brazil.
	}
	\email{marco.a.fernandes@ufv.br}
        \author[]{M. Salarinoghabi}
	\address[Mostafa]{Departamento de Matem\'atica, Universidade Federal de Vi\c cosa, Brazil.
	}
	\email{mostafa.salarinoghabi@ufv.br }
	
	\subjclass[2020]{%
		Primary 53A04; 
		Secondary 53A35, 58K05 
		. 
	}

	\keywords{Evolutoids, wavefronts, involutoids, non-euclidean geometry.}
\begin{abstract}
Let $M_c$ be a $2$-dimensional space form of constant curvature $c=-1,0,1$ and $\gamma$ a smooth, closed, convex curve in $M_c$. We explicitly parametrize the \textit{$\alpha$-evolutoid} of $\gamma$, i.e.\ the closed curve $\gamma_\alpha$ describing the envelope of all geodesics $\sigma_s=\sigma_s(t)$ such that $\sigma_s(0)=\gamma(s)$ and $\sphericalangle(\sigma_s'(0),\gamma'(s))=\alpha$, with $\alpha\in[0,\pi/2]$ fixed and determine its lenght. Also, we deduce that for each $s$ the points $\gamma(s),\gamma_\alpha(s),\gamma_{\pi/2}(s)$ belong to a distinct geodesic circle. A constraint for the smoothness of $\gamma_\alpha$ is calculated and, using tools from singularity theory, we prove that its singularities present cuspidal features, which mimics the classical evolute ($\alpha=\pi/2$) in the plane case. Also, we define the \textit{$\alpha$-involutoids} of a given curve $\eta$ in $M_c$ to be any curve $\gamma$ in $M_c$ such that $\gamma_\alpha=\eta$ and study some of its properties. In particular, we prove that any convex, closed curve in $M_{-1,0}$ has associated to itself exactly one closed $\alpha$-involutoid. Finally, we show that the evolutoids can be seen as singular sets of wavefronts. 
\end{abstract}

\maketitle


\section{Introduction}
Given a smooth, convex, closed curve $\gamma:[0,\ell]\longrightarrow\mathbb{R}^2$, its  \textit{evolute} is the envelope of all its normal lines \cite{bruce1992curves}, which we denote by $\gamma_{\pi/2}$, and has been studied systematically since at least Huygens \cite{huygens1966horologium}. Let $M_c$ be a $2$-dimensional space form of constant curvature $c=-1,0,1$ and $\gamma$ a smooth, closed, convex curve in $M_c$. We study the \textit{$\alpha$-evolutoid} of $\gamma$, i.e.\ the closed curve $\gamma_\alpha$ describing the envelope of all geodesics $\sigma_s=\sigma_s(t)$ such that $\sigma_s(0)=\gamma(s)$ and $\sphericalangle(\sigma_s'(0),\gamma'(s))=\alpha$, with $\alpha\in[0,\pi/2]$ fixed. The classical evolute of a plane curve $\gamma$ is closely related to its \textit{focal set} \cite{bruce1992curves}.

Many features of the plane evolute were extended to 2-dimensional space forms. On the integral geometry side, several characteristics of plane, convex curves and some of its associated curves (like the evolute and \textit{wavefronts}) have been investigated on space forms since the 2000's by Solanes and others (\cite{gallego2004horospheres}, \cite{solanes2005integral2}, \cite{solanes2005integral} and \cite{solanes2006integral}). In \cite{escudero2007-Pacific}, for example, the authors extend a result from the plane geometry which relates the area of a convex set in a 2-dimensional space of constant curvature with some integrals of the curvature radius at its boundary, in a space-form analogue of Heintze and Karcher's inequality. Jerónimo-Castro \cite{jeronimo2014-Aequationes} showed that in the plane case, if we denote by $A(\gamma$) the area enclosed by $\gamma$, then for sufficiently small $\alpha$ the area enclosed by the $\alpha$-evolutoid of $\gamma$ satisfies
\begin{equation}\label{areacomparison}
	\frac{A_0(\gamma_\alpha)}{A_0(\gamma)}\leqslant\cos^2(\alpha),
\end{equation}
with equality holding (in this case for \textit{any} $\alpha\in[0,\pi/2])$ if and only if $\gamma$ is a circle. For this task, he uses the so-called \textit{support functions} in order to suitably describe and investigate $\gamma$ and $\gamma_\alpha$ via Blasche and Cauchy formulae for the area enclosed by a curve. Of course such a type of inequality would be an interesting result in $M_{-1}$, but some properties of plane support functions depend strongly on the Euclid's 5\textsuperscript{th} Postulate. Attempts to define \textit{hyperbolic} support functions have been made in \cite{leichtweiss2004support}. Still, a hyperbolic version of \eqref{areacomparison} seems to require the replacement of $\cos^2(\alpha)$ by a function $f=f(\alpha,\diam(\gamma))$. Despite the difficulties in extending these results to \textit{area}, we obtain the \textit{length} of $\gamma_\alpha$ for a given closed convex curve $\gamma:[0,\ell]\longrightarrow M_{-1,1}$ (Proposition \ref{teo:comprimeinto-evolutoide}) and draw a bit more attention to these questions in Section \ref{section:area}.

Still concerning metric properties of $\gamma_\alpha$, we prove that for a given closed, smooth, and convex curve $\gamma$ in $M_{-1,1}$, the points $\gamma(s)$, $\gamma_\alpha(s)$ and $\gamma_{\pi/2}(s)$ belong to a geodesic disc of diameter equals to the curvature radius of $\gamma$ at $\gamma(s)$, for all $\alpha\in[0,\pi/2]$ (Corollary \ref{giblinhyperbolic}), as an extension of the plane case, although in this case such a property is linked to Euclid's 5\textsuperscript{th} Postulate.

From the viewpoint of Singularity Theory, the \textit{hyperbolic evolute} of a given closed convex curve $\gamma\subset M_{-1}$ has been investigated by Izumyia and others in \cite{izumiya2004-A4}, where the authors state relations between its singularities and certain geometric invariants of curves under the action of the Lorentz group. On the plane case, it is fairy known \cite{giblin2014-B2} that the \textit{cusps} of the evolute correspond to the critical points of the geodesic curvature of $\gamma$ (the so-called \textit{vertices} of $\gamma$, which appear at least $4$ times if $\gamma$ is closed). In \cite{izumiya2004-A4}, the authors extend this result linking the cusps of the hyperbolic evolute $\gamma_{\pi/2}$ of $\gamma$ to its vertices using unfoldings from Singularity Theory, and they provided an explicit parametrization of $\gamma_{\pi/2}$. With our parametrization of $\gamma_\alpha$, we obtain conditions on $\alpha$ which ensure the smoothness of $\gamma_\alpha$ for a given $\gamma$ (Section \ref{evolutoids}, Theorem \ref{teo:evolutoideregularidade}).

This work is organized as follows: Section 2 sets the basic notations and facts we use throughout the document. Section 3 deals with the description, differential geometry and singularities of the $\alpha$-evolutoids. Section \ref{section:area} discuss some aspects and natural questions concerning area and lenght of the $\alpha$-evolutoids in space forms. In Section \ref{secaoinvolutoides}, we define the $\alpha$-involutoids for a given smooth, closed curve $\gamma:[0,\ell]\longrightarrow M_c$ to be any curve $\eta$ which is the $\alpha$-evolutoid of $\gamma$ and prove that amongst the (infinite) family of $\alpha$-involutoids of such a given curve, there is always a unique closed one, extending the results from \cite{jeronimo2014-Aequationes}. Section \ref{wavefronts} contains a short characterization of the $\alpha$-evolutoids of a given closed, convex curve $\gamma:[0,\ell]\longrightarrow M_c$ as singular points of the \textit{wavefront} of $\gamma$ in the direction
$$\vec{v}_\alpha(s)=\cos(\alpha)\vec{t}(s)+\sin(\alpha)\vec{e}(s),$$ where $\{\vec{t},\vec{e}\}$ is  the natural orthonormal frame along $\gamma$. On this theme, a description of the \textit{singular curvature function} on cuspidal edges of surfaces on $\mathbb{R}^3$ have been linked to wavefronts and published by Saji, Yamada and Umehara \cite{saji2009geometry}.


\section{Preliminaries and notations}

Recurrently along this work, we chose $\alpha\in[0,\pi/2]$. So we set
\begin{equation}\label{aeb}
	a=\cos(\alpha)\hspace{1cm}\text{and}\hspace{1cm}b=\sin(\alpha).
\end{equation}
In this technical section, we extend in a natural way some useful results from \cite{izumiya2004-A4} to the case of space forms.  Let $c=-1,0,1$ and $M_c$ be the two-dimensional space forms of constant Gaussian curvature $c$. In Subsections \ref{esferaehiperboloide} and \ref{plano}, respectively, we briefly describe the Riemannian geometry of $M_c$ when $c=-1,1$ and $c=0$. At the end of Subsection \ref{plano}, we provide a Frenet-Serret frame for unit-speed curves $\gamma:I\longrightarrow M_c$ with $c=-1,0,1$. In Section \ref{unifiednotations} we describe some unified notations we shall use along the text.

Consider the vector space $\R^3=\{\vec{x}=(x_1,x_2,x_3)\,|\,x_1,x_2,x_3\in\R\}$ endowed with a symmetric bilinear $2$-form given by
\begin{equation}\label{eq:minkowskymetric}
\left\langle \vec{x},\vec{y}\right\rangle_c=cx_1y_1+x_2y_2+x_3y_3.	
\end{equation}
 We set $\|\vec{x}\|_c=\sqrt{\left\langle\vec{x},\vec{x}\right\rangle_c}$ for any $\vec{x}\in\R^3$.

\subsection{The hyperbolic and spherical cases}\label{esferaehiperboloide}

Let $c=-1,1$. We shall use the classical Minkowsky model of the hyperbolic space for $M_{-1}$, i.e.\ the 1-sheet hyperboloid
\[
\H^2=\{\vec{x}\in\R^3\,|\,\left\langle\vec{x},\vec{x}\right\rangle_{-1}=-1\,|\,x_1>0\}
\]
with the metric (of constant negative Gaussian curvature) obtained by restricting $\left\langle\cdot,\cdot\right\rangle_{-1}$ to the tangent planes of $M_{-1}$. For $M_1$ we likewise use the standard sphere 
\[
\mathbb{S}^2=\{\vec{x}\in\R^3\,|\,\left\langle\vec{x},\vec{x}\right\rangle_1=1\},
\]
with the Euclidean metric $\left\langle\cdot,\cdot\right\rangle_1$ inducing positive constant curvature on $\mathbb{S}^2$. If $\vec{x},\vec{y}\in\R^3$, we define
\[\label{eq:wedgeproduct}
\vec{x}\wedge_c\vec{y}=(c(x_2y_3-x_3y_2),x_3y_1-x_1y_3,x_1y_2-x_2y_1).
\]
In particular, if $\vec{x},\vec{y},\vec{z}\in\R^3$, then \eqref{eq:minkowskymetric} gives $\left\langle\vec{x}\wedge_c\vec{y},\vec{z}\right\rangle_c=\det(\vec{x}\,\vec{y}\,\vec{z})$.

Let $\gamma=\gamma(s)$ be an arc-lenght parametrized curve in $M_c$ and consider $\vec{t}(s)=\gamma'(s)$, $\vec{e}(s)=\gamma(s)\wedge_c\vec{t}(s)$. We summarize some properties of the frame $\{\gamma(s),\vec{t}(s),\vec{e}(s)\}$, following \cite{izumiya2004-A4}, Theorem 2.1:

\begin{lem}\label{lemma:frenetserretframedefinition}If $c=-1,1$, then the frame $\{\gamma(s),\vec{t}(s),\vec{e}(s)\}$ satisfies
	\begin{enumerate}
		\item[i)] $\left\langle\gamma(s),\vec{t}(s)\right\rangle_c=0,$
		\item[ii)] $\left\langle\vec{e}(s),\vec{e}(s)\right\rangle_c=1$,
		\item[iii)] $\vec{t}(s)\wedge_c\vec{e}(s)=c\gamma(s)$,
		\item[iv)] $\vec{e}(s)\wedge_c\gamma(s)=\vec{t}(s)$,
	\end{enumerate}
for all $s$.
\end{lem}
\begin{proof}It is a straightforward calculation.
\end{proof}

The $3$-frame $\{\gamma(s),\vec{t}(s),\vec{e}(s)\}$ will be used to study the local geometry of an arc-parametrized curve $\gamma$ in $M_c$. Let $\nabla$ indistinctly denote the Levi-Civita connection of $(M_c,\left\langle\cdot,\cdot\right\rangle_c)$. Given any regular, arc-parametrized curve $\gamma:I\subset\R\longrightarrow M_c$, we define its geodesic curvature $k$ via
\begin{equation}\label{eq:geodesiccurvature}
	\nabla_{\gamma'}\gamma'=k\vec{e},
\end{equation}
where $\vec{e}(s)=\gamma(s)\wedge_c\vec{t}(s)$. If $\gamma:I\subset\R\longrightarrow M_c$ does not have unit-speed, it is well-known that its geodesic curvature $k=k(s)$ is given by
\begin{equation}\label{geodesic-curvature-formula}
	k(s)=\frac{\det(\gamma(s)\,\,\vec{t}(s)\,\,\vec{e}(s))}{\|\vec{t}(s)\|_c^3}=\frac{\langle \gamma(s)\wedge_c\vec{t}(s),\vec{e}(s)\rangle_c}{\|\vec{t}(s)\|_c^3}.
\end{equation}

\begin{prop} Let $\gamma:I\subset\R\longrightarrow M_c$ an arc-length parametrized curve. With the above notations, the Frenet-Serret formulas for $\gamma$ are
	\begin{equation}\label{eq:frenetserretframe}
		\begin{cases}
			\gamma'(s)=\vec{t}(s) \\
			\vec{t}'(s)=-c\gamma(s)+k(s)\vec{e}(s) \\
			\vec{e}'(s) = -k(s)\vec{t}(s)
		\end{cases}
	\end{equation}
\end{prop}
\begin{proof}The proof is straightforward using Lemma \ref{lemma:frenetserretframedefinition} and following the steps in Theorem 2.1 from \cite{izumiya2004-A4}.
\end{proof}

\subsection{The plane case}\label{plano} Let $c=0$. Our model for the plane $M_0$ will be
$$\R^2\equiv M_0=\{\vec{x}\in\R^3\,|\,x_1=0\}$$
with the flat metric given by the pull-back of $\left\langle\cdot,\cdot\right\rangle_0$ to $M_0$. We would like to obtain formulas like \eqref{eq:frenetserretframe} for an arc-length parametrized curve $\gamma$ in $M_0$. To accomplish this task, let $\vec{t}=\gamma'$. By observing that $\left\langle\gamma',\gamma''\right\rangle_0=0$, we denote $\vec{e}=\gamma''/\|\gamma''\|_0$ and then the geodesic curvature $k$ of $\gamma$ satisfies the well-known relations
\begin{equation}
	\begin{cases}
		\gamma'(s)=\vec{t}(s) \\
		\vec{t}'(s)=k(s)\vec{e}(s) \\
		\vec{e}'(s) = -k(s)\vec{t}(s).
	\end{cases}
\end{equation} 

\subsection{Frenet-Serret formulae}With the notations introduced above, we have the following

\begin{prop}\label{teo:unifiedfrenetframes}Let $\gamma:I\subset\R\longrightarrow M_c$ a regular, arc-length parametrized curve and $c=-1,0,1$. Consider the normal vector field along $\gamma$ given by
\begin{equation}\label{eq:normalvectorfield}
	\vec{e}(s)=\begin{cases}
		\gamma(s)\wedge_c\vec{t}(s),& c=-1,1 \\
		\gamma''(s)/\|\gamma''(s)\|_0, &  c=0.
	\end{cases}
\end{equation}
	Then the Frenet-Serret frame $\{\gamma(s),\vec{t}(s),\vec{e}(s)\}$ satisfies
		\begin{equation}\label{eq:unifiedfrenetframes}
			\begin{cases}
				\gamma'(s)=\vec{t}(s) \\
				\vec{t}'(s)=-c\gamma(s)+k\vec{e}(s) \\
				\vec{e}'(s) = -k\vec{t}(s)
			\end{cases}
		\end{equation}
\end{prop}
\begin{proof}It is a straightforward calculation.
\end{proof}

Notice that $\{\vec{t},\vec{e}\}$ is an \textit{intrinsic} orthonormal frame for $M_c$ with the induced Riemannian metric $\left\langle\cdot,\cdot\right\rangle_c$, although $\{\gamma,\vec{t},\vec{e}\}$ is not an \textit{actual} (i.e.\ linearly independent) frame on $\R^3$ when $c=0$. We describe some generalized notations below and, for convenience, we'll use them even when $c=0$, unless some particularity must be pointed out.

\subsection{Unified notations and exponential maps}\label{unifiednotations}

We recall the exponential maps of $M_c$, using the notations from \cite{escudero2007-Pacific} in order to be concise.  Let $c=-1,0,1$. The trigonometric functions to be used in this work can be written at once as
\begin{equation}\label{eq:senos}
	\sin_c(t)=\begin{cases}
		\sinh(t), & c=-1\\
		t, & c=0\\
		\sin(t), & c=1
	\end{cases}
\end{equation}
and
\begin{equation}\label{eq:cossenos}
	\cos_c(t)=\begin{cases}
		\cosh(t), & c=-1\\
		1, & c=0\\
		\cos(t), & c=1,
	\end{cases}
\end{equation}
which lead us to define $\tan_c(t)=\sin_c(t)/\cos_c(t)$, $\cot_c(t)=1/\tan_c(t)$ 
and $\textrm{arccot}_c(t)=\cot_c^{-1}(t)$ etc. In particular, we have some known trigonometric identities:

\begin{table}[h!]
	\centering
	\begin{tabular}{ll}
		$c \sin_c^2(t) + \cos_c^2(t) = 1,$ & $\cos_c'(t) = -c \sin_c(t),$ \\
		$\cos_c(2t) = \cos_c^2(t) - c \sin_c^2(t),$ & $\sin_c'(t) = \cos_c(t)$.
	\end{tabular}
\end{table}

\noindent These notations are useful when dealing simultaneously with all the three space forms $M_c$, since the exponential map of the Riemannian surface $(M_c,\left\langle\cdot,\cdot\right\rangle_c)$ can be easily expressed as
\begin{equation}\label{eq:exponentialmaps}
	\exp^c_p(tv)=\cos_c(t) p+\sin_c(t) v,
\end{equation}
\noindent where  $p\in M_c$, $v\in T_pM_c$ is a unitary tangent vector and $t\in\R$. We say that a given smooth, closed, positively-oriented, arc-parametrized curve $\gamma:[0,\ell]\longrightarrow M_c$ is \emph{convex} when 
\begin{equation}\label{eq:convexcurves}
	\begin{cases}
		k(s)>0, & c=0,1\\
		k(s)>1, & c=-1
	\end{cases}
\end{equation}
for all $s\in[0,\ell]$. The second condition in \eqref{eq:convexcurves} is in fact used to define \textit{$h$-convexity} in \cite{escudero2007-Pacific}. The above inequalities may be geometrically viewed as follows: if we consider geodesic circles of radius $R\rightarrow\infty$ in $\R^2$ or $\H^2$ and $R\rightarrow\pi/2$ in $\mathbb{S}^2$, the limits of its geodesic curvatures are 0, 1 and 0 respectively. 


\section{Evolutoids of convex curves and its singularities}\label{evolutoids}

Let $c=-1,0,1$. Consider $\gamma:[0,\ell]\longrightarrow M_c$ a smooth, closed, convex curve and  $\{\vec{t}(s),\vec{e}(s)\}$ its positive orthonormal Frenet-Serret frame, as in Proposition \ref{teo:unifiedfrenetframes}. The envelope of all geodesics $\sigma_s:[0,\infty)\longrightarrow M_c$ satisfying $\sigma_s(0)=\gamma(s)$ and $\sigma_s'(0)=\vec{e}(s)$ is the so-called \emph{evolute} of $\gamma$ (\cite{bruce1992curves}, \cite{escudero2007-Pacific}, \cite{izumiya2004-A4}). From now on, we fix $\alpha\in[0,\pi/2]$ and set, for $(t,s)\in[0,\ell]\times[0,\infty)$, the family of geodesics given by

\begin{equation}\label{2parameterfamily}
	\F_\alpha(s,t)=\exp^c_{\gamma(s)}(t \vec{v}_\alpha(s)),
\end{equation}
where
\begin{equation}\label{eq:valpha}
	\vec{v}_\alpha(s)=a\vec{t}(s)+b\vec{e}(s),
\end{equation}

\noindent see \eqref{aeb}. Geometrically, $\F_\alpha$ is a family of slanted geodesics all intersecting $\gamma$ with angle $\alpha$, with $\alpha$ positively oriented with respect to $\{\vec{t},\vec{e}\}$, so that their velocities at $\gamma(s)$ point ``inward'' $\gamma$.

\subsection{Evolutoids of convex curves in space forms}
The envelope of the family $\F_\alpha$ is called the \emph{$\alpha$-evolutoid} of $\gamma$ and will be denote by $\gamma_\alpha$. If $c=-1$ and $\alpha=\pi/2$, we obtain the \textit{hyperbolic evolute} \cite{izumiya2004-A4} of $\gamma$. Adopting the techniques from \cite{escudero2007-Pacific}, who describes the \emph{evolute} on a space form $M_c$ as the set of singularities of the family $\mathcal{F}_{\pi/2}(s,\cdot)$, we set the following

\begin{definition}\label{def:evolutoide} Fix $\alpha\in[0,\pi/2]$. The \emph{$\alpha$-evolutoid} of a closed, convex, arc-parametrized, regular curve $\gamma:[0,\ell]\longrightarrow M_c$ is the image by \eqref{2parameterfamily} of the critical set of the family $\F_\alpha(s,\cdot)$.
\end{definition}

\begin{figure}[h!]
    \centering
    \begin{subfigure}[b]{0.2\textwidth}
        \includegraphics[width=\textwidth]{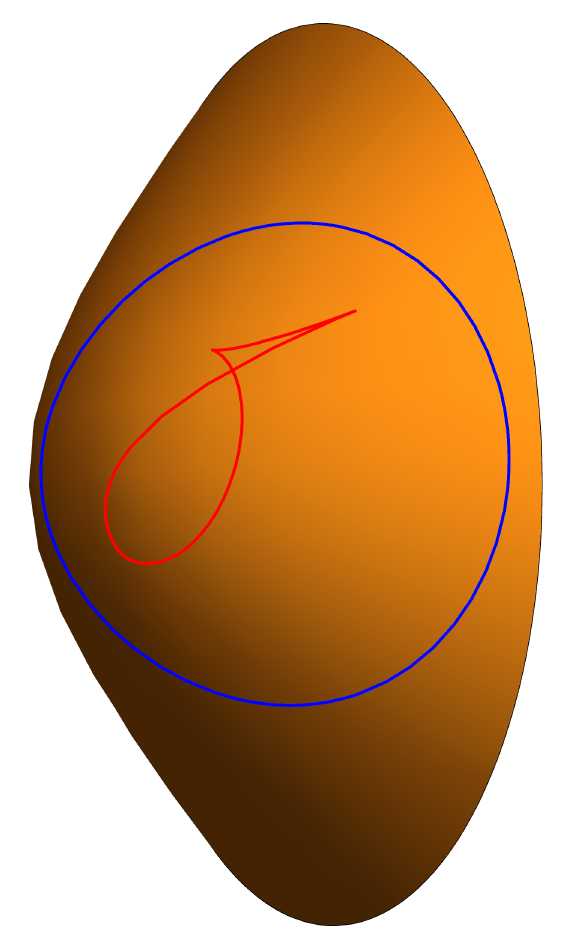}
        \caption{$c=-1$}
    \end{subfigure}
    \quad
    \begin{subfigure}[b]{0.3\textwidth}
        \includegraphics[width=\textwidth]{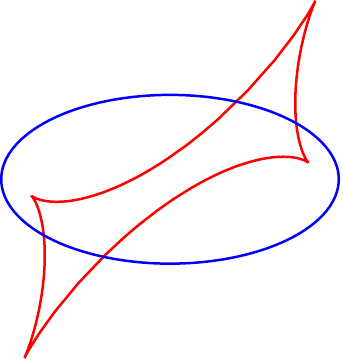}
        \caption{$c=0$}
    \end{subfigure}
    \quad
    \begin{subfigure}[b]{0.3\textwidth}
        \includegraphics[width=\textwidth]{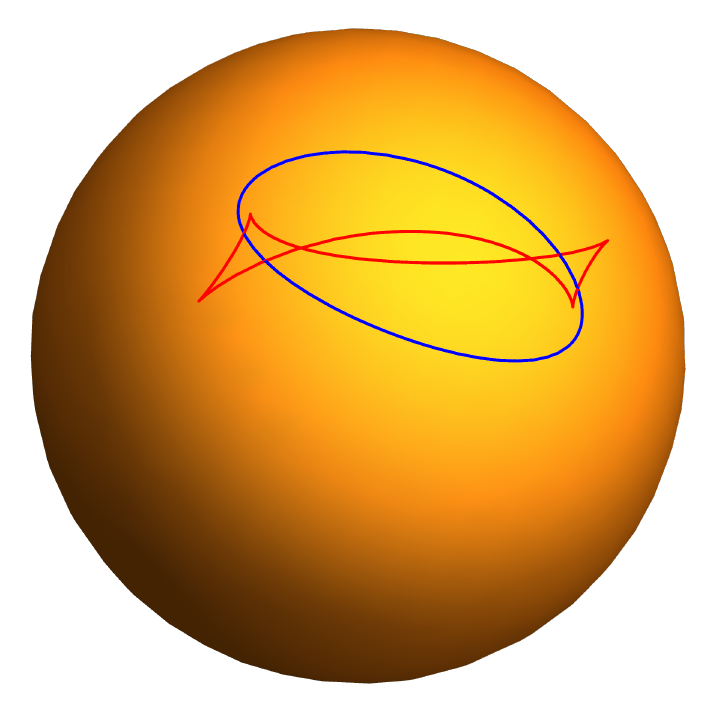}
        \caption{$c=1$}
    \end{subfigure}
    \caption{Examples of $\alpha$-evolutoids of closed curves. The original curve is shown in blue and the $\alpha$-evolutoid in red.}
\end{figure}

Since in $M_{-1,0}$ the geodesics meet at most once, then for $c=-1,0$ there exists $\rho_\alpha:[0,\ell]\longrightarrow\R$ such that each $(D\F_\alpha)_{(s,\rho_\alpha(s))}$ is a singular linear map, as if the $\alpha$-evolutoid were a \textit{graph} over $\gamma$. On the sphere, the closedness of all geodesics imply that the family $\mathcal{F}(s,\cdot)$ has periodic singularities and the $\alpha$-evolutoid has actually \textit{two} isometric connected components, symmetric to each other with respect to $(0,0,0)$ (see Remark \ref{sphericalevolutoids}). Convex spherical curves lie in hemispheres (\cite{toponogov2006differential}, Problems 1.10.1-4), allowing us to consider only one copy of its evolutoids.



From now on, we can assume that for all smooth, convex, closed curves $\gamma$ in space forms, there exists $\rho_\alpha:[0,\ell]\longrightarrow\R$ such that each $(d\F_\alpha)_{(s,\rho_\alpha(s))}$ is a singular linear map. We henceforward study some properties of this function and its consequences.

The $\alpha$-evolutoids are closed, piecewise regular curves $\gamma_\alpha=\gamma_\alpha(s)$ and have been studied in the plane and hyperbolic spaces, see \cite{giblin2014-B2}, \cite{Izumyia2016} and \cite{escudero2007-Pacific}, for instance. We provide an explicit parametrization of the $\alpha$-evolutoid of a curve $\gamma$ satisfying the conditions of Definition \ref{def:evolutoide} (see Figure \ref{fig:evolutoid}).

\begin{prop}\label{teo:parametrizacao-evolutoide}
	The function $\rho_\alpha:[0,\ell]\longrightarrow M_c$ is given by
	\begin{equation}\label{eq:raiodecurvaturainclinado2}
		\rho_\alpha(s)=\arctan_c\left(\frac{b}{k(s)}\right).
	\end{equation} 
	for all $s\in[0,\ell]$. That is, the $\alpha$-evolutoid of a smooth, convex, closed curve $\gamma=\gamma(s)$ in $M_c$ is parametrized by
	\begin{equation}\label{eq:parametrizacao-evolutoide}
		\gamma_\alpha(s)=\exp^c_{\gamma(s)}(\rho_\alpha(s)\vec{v}_\alpha(s)).
	\end{equation}
\end{prop}
\begin{proof} We calculate the derivative of $\F$ by using \eqref{eq:exponentialmaps}, omitting variables on the Frenet-Serret frame:
	$$
		\frac{\partial\F_\alpha}{\partial_s }(s,t) = -ac\sin_c(t)\gamma + (\cos_c(t)-bk\sin_c(t))\vec{t} + ak\sin_c(t)\vec{e} $$
	and
		$$\frac{\partial\F_\alpha}{\partial_t }(s,t) =  -c\sin_c(t)\gamma+a\cos_c(t)\vec{t}+b\cos_c(t)\vec{e}.$$
Writing the derivative of $\mathcal{F}_\alpha$ w.r.t.\ to $\{(1,0),(0,1)\}$ and $\{\gamma,\vec{t},\vec{e}\}$,
\begin{equation}\label{eq:derivadadafamiliaF}
	(D\F_\alpha)_{(s,t)}=\left[\begin{array}{cc}
		-ac\sin_c(t)			&  -c\sin_c(t) \\
		\cos_c(t)-bk\sin_c(t)	&  a\cos_c(t)  \\
		ak\sin_c(t)				&  b\cos_c(t)
	\end{array}\right].
\end{equation}
If $c=0$, the first row of \eqref{eq:derivadadafamiliaF} vanishes and $(D\mathcal{F}_\alpha)_{(s,t)}$ does not have full rank if and only if $\rho_\alpha(s)=t = b/k(s)$, i.e. $\rho_\alpha(t)=\mathrm{arcot}_0(k(s)/b)$, giving us \eqref{eq:raiodecurvaturainclinado2}. If $c=-1,1$, $(D\F_\alpha)_{(s,t)}$ has non-trivial kernel if and only if 
\begin{eqnarray}\label{eq:singularidadesdafamiliaF}
\frac{\partial\F_\alpha}{\partial s}\wedge_c\frac{\partial\F_\alpha}{\partial t } &=& c(b\cos_c(t)-k\sin_c(t))(\cos_c(t),a\sin_c(t),b\sin_c(t))\nonumber\\
	& = &(0,0,0),
\end{eqnarray}
which happens exactly when \eqref{eq:raiodecurvaturainclinado2} holds.
\end{proof}
\begin{figure}[h!]
	\centering
    \includegraphics[width=0.75\textwidth]{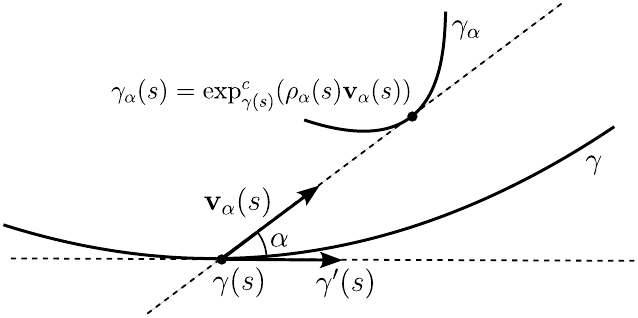}
    \caption{$\alpha$-evolutoid of a curve $\gamma$.}
    \label{fig:evolutoid}
\end{figure}

\begin{rem}\label{sphericalevolutoids}The spherical version of \eqref{eq:raiodecurvaturainclinado2} satisfies
$$\tan(\rho_\alpha(s))=\tan(\rho_\alpha(s)+n\pi)=\frac{b}{k(s)},$$
for all $n\in\mathbb{N}$. This implies that the $\alpha$-evolutoid on a sphere is not a graph over $\gamma$, but two symmetric copies -- relatively to $(0,0,0)$ -- of the same graph $(s,\rho_\alpha(s))$ over $\gamma$, a redundancy avoided by setting $\rho_\alpha(s)$ to be the smallest nonnegative number that makes $(D\mathcal{F}_\alpha)_{(s,\cdot)}$ singular.
\end{rem}

We explore now some geometric properties of the $\alpha$-evolutoids. Giblin (\cite{giblin2014-B2}, Remark 2.4(i)) pointed out that in the plane case the segments
\begin{center}
$\overline{\gamma(s)\gamma_\alpha(s)}$\hspace{1cm} and \hspace{1cm}$\overline{\gamma_\alpha(s)\gamma_{\pi/2}(s)}$	
\end{center}
\noindent are orthogonal at $\gamma_\alpha(s)$  for each $s$ or, equivalently, $\gamma(s)$, $\gamma_\alpha(s)$ and $\gamma_{\pi/2}(s)$ belong to a distinct geodesic circle of diameter $1/k(s)$, for all $\alpha\in[0,\pi/2]$, which is clearly tangent to $\gamma$ at $\gamma(s)$. Although the perpendicularity between $\overline{\gamma(s)\gamma_\alpha(s)}$ and $\overline{\gamma_\alpha(s)\gamma_{\pi/2}(s)}$ should not be expected in $M_{-1,1}$ (since these manifolds lack Euclid's 5\textsuperscript{th} Postulate), we can drop the perpendicularity condition while keeping the points $\gamma(s)$, $\gamma_\alpha(s)$ and $\gamma_{\pi/2}(s)$ belonging to a geodesic circle when $c=1,-1$. The case $c=-1$ will be adressed now. The argument for $c=1$ will be outlined right after.

We recall, for convenience, some features of the Poincaré's hyperbolic metric on $D(0,1)=\{z\in\mathbb{C}\vcentcolon|z|<1\}$. If we denote by $d_{\text{hyp}}$ and $d_{\text{euc}}$ the hyperbolic and euclidean distances on $D(0,1)$, then the following well known relation holds: for all $z\in D(0,1)$,
\begin{equation}\label{euclidtohyperbolic}
	d_\text{hyp}(z,0)=\arctanh\big(d_\text{euc}(z,0)\big).
\end{equation}
Let $C\subset D(0,1)$ be a geodesic circle of hyperbolic diameter $d$ and $L$ be a straight line from 0 making angle $\alpha\in(0,\pi/2)$ with the real positive axis. Assume $C$ lies on the upper half plane and is tangent to the real axis at $0$. Then $C$ intersects $L$ at a single point $Q=x+iy$ with $y>0$.

\begin{lem}\label{giblinhyperboliclemma}Let $\{Q\}=L\cap C\cap\{y>0\}$. Then $$d_\textnormal{hyp}(Q,0)=\arctanh\big(\tanh(d)\sin(\alpha)\big).$$
\end{lem}
\begin{proof}By \eqref{euclidtohyperbolic}, $C$ has euclidean diameter $\tanh(d)$ and hence has cartesian equation
	 $$x^2+\left(y-\frac{\tanh(d)}{2}\right)^2=\frac{\tanh^4(d)}{4}.$$
The point $(t\cos(\alpha),t\sin(\alpha))$  that solves the above equation for $t>0$ does so when $t=\tanh(d)\sin(\alpha)=d_\text{euc}(Q,0).$ By \eqref{euclidtohyperbolic} again, we have the result.
\end{proof}

Now we proceeed to prove a hyperbolic analogous to Giblins' observation on the plane case (\cite{giblin2014-B2}, Remark 2.4(i)).

\begin{coro}\label{giblinhyperbolic} Let $\gamma:[0,\ell]\longrightarrow D(0,1)$ be a smooth, convex, positively oriented, closed curve. The points $\gamma(s)$, $\gamma_\alpha(s)$ and $\gamma_{\pi/2}(s)$ lie in the same geodesic circle of diameter $\rho_{\pi/2}(s)$ for all $\alpha\in[0,\pi/2]$. Furthermore, this circle is tangent to $\gamma$ at $\gamma(s)$.
\end{coro}
\begin{proof}Without loss of generality, we can assume $s=0$, $\gamma(0)=0$ and $\gamma'(0)=(1,0)$. Consider the circle $C$ to be tangent to the real axis at $0$ and with diameter $\rho_{\pi/2}(0)$, lying in the upper half plane. Lemma \ref{giblinhyperboliclemma} shows that the (non-unitary) hyperbolic geodesic $(t\cos(\alpha),t\sin(\alpha))$ intersects $C$ at a point $Q$ satisfying
	$$d_\text{hyp}(Q,\gamma(0))=\arctanh\left(\tanh(\rho_{\pi/2}(0))\sin(\alpha)\right)=\rho_\alpha(0),$$
and we conclude $Q=\gamma_{\alpha}(0)$.
\end{proof}

We outline the argument for the spherical case: the above techniques can be applied if one consider the stereographical projection $P$ from (the north pole of) $$\mathbb{S}^2_{(0,0,1)}\vcentcolon=\{(x,y,z)\in\R^3\,|\,x^2+y^2+(z-1)^2=1\}-\{(0,0,2)\}$$
onto $\{(x,y,0)\,|\,x,y\in\mathbb{R}\}$ and use the relation
$$d_\text{euc}(Q',(0,0,0))=2\tan\left(\frac{d_\text{sph}(Q,(0,0,0))}{2}\right),$$ in which $Q'=P(Q)$ and $d_\text{sph}$ denotes, likewise, the spherical distance.

\subsection{Smoothness of the $\alpha$-evolutoid}
The evolute $\gamma_{\pi/2}$ itself has well understood singularities in the plane case \cite{giblin2014-B2}. For the hyperbolic case, these singularities and other geometric features have been studied in \cite{escudero2007-Pacific}, \cite{izumiya2004-A4}. On the unit sphere, Nishimura \cite{nishimura2008classification} investigate \textit{pedal curves}, which are naturally linked to the evolutes. In both cases, the critical points of $k=k(s)$ give rise to the (cuspidal) singularities of $\gamma_\alpha$. Since for small $\alpha\in[0,\pi/2)$ we see that $\gamma_\alpha$ is regular (diffeomorphically homotopic to $\gamma$), we would like to adress the problem of finding the \emph{smallest} $\alpha\in(0,\pi/2]$ such that $\gamma_\alpha$ is \emph{not} regular -- i.e., $\gamma'_\alpha(s)=0$ for \emph{some} $s\in[0,\ell]$ -- and qualify features these singularities present.

From equation \eqref{eq:parametrizacao-evolutoide} and the chain rule, the smoothness classes of $\gamma_\alpha$ and of $\rho_\alpha$ are related. If we consider $c=-1$, it has been proved (\cite{izumiya2004-A4}, Theorem 5.3) that the singular points of the hyperbolic evolute $\gamma_{\pi/2}$ occur exactly for those $s$ such that $\gamma(s)$ is a \textit{vertex} of $\gamma$. This mimics the well-known plane case. With this perspective, we prove below an extension of this result.

\begin{teo}\label{teo:evolutoideregularidade} Fix $\alpha\in(0,\pi/2]$ and let $\gamma:[0,\ell]\longrightarrow M_c$ be a smooth, arc-legth parametrized, convex, closed, positively oriented curve. The $\alpha$-evolutoid of $\gamma$ given by \eqref{eq:parametrizacao-evolutoide} is regular if and only if
\begin{eqnarray}\label{eq:evolutoideregularidade}
	\rho'_\alpha(s)\neq-\cos(\alpha).
\end{eqnarray}
\end{teo}
\begin{proof}Recall that $a=\cos(\alpha)$ and $b=\sin(\alpha)$. From \eqref{eq:exponentialmaps},
	$$\gamma_\alpha(s)=\cos_c(\rho_\alpha(s))\gamma(s)+\sin_c(\rho_\alpha(s))(a\vec{t}(s)+b\vec{e}(s)).$$
	We derive $\gamma_\alpha=\gamma_\alpha(s)$ by using \eqref{eq:unifiedfrenetframes}, \eqref{eq:senos}, \eqref{eq:cossenos} and omitting $s$:
	\begin{eqnarray}\label{eq:regularidadeevolutoides}
		\gamma'_\alpha &=&-c\sin_c(\rho_\alpha)(\rho_\alpha'+a)\gamma\nonumber\\
		& &+ ((1+a\rho_\alpha')\cos_c(\rho_\alpha)-b\sin_c(\rho_\alpha)k)\vec{t}\\
		& &+ (a\sin_c(\rho_\alpha)k+b\cos_c(\rho_\alpha)\rho_\alpha')\vec{e}\nonumber.
	\end{eqnarray}
Let $c=-1,1$. Using \eqref{eq:raiodecurvaturainclinado2} in the form $b\cos_c(\rho_\alpha)=k\sin_c(\rho_\alpha)$ on \eqref{eq:regularidadeevolutoides}, we readily get
\begin{equation}\label{derivadaevolutoide}
	\gamma_\alpha'=(\rho_\alpha'+a)\big[-c\sin_c(\rho_\alpha)\gamma+\cos_c(\rho_\alpha)\vec{v}_\alpha\big]
\end{equation}
It is easy to see that $\|-c\sin_c(\rho_\alpha)\gamma+\cos_c(\rho_\alpha)\vec{v}_\alpha\|_c=1$. Then
we conclude condition \eqref{eq:evolutoideregularidade}. Now let $c=0$. By \eqref{eq:raiodecurvaturainclinado2}, we have $\rho_\alpha(s)=b/k(s)$. Using \eqref{eq:senos} and \eqref{eq:cossenos}, equation \eqref{eq:regularidadeevolutoides} becomes
$$\gamma_\alpha'=(a^2+a\rho_\alpha')\vec{t} + (ab+b\rho_\alpha')\vec{e}=(\rho_\alpha'+a)\vec{v}_\alpha,$$
which vanishes exactly if and only if $\rho_\alpha'(s)=-a$.
\end{proof}
\begin{rem}It is worth noticing that the vector field $$-c\sin_c(\rho_\alpha)\gamma+\cos_c(\rho_\alpha)\vec{v}_\alpha$$
in \eqref{derivadaevolutoide} is unitary, tangent  to $\gamma$ and  canonically associated with the parametrization $\gamma=\gamma(s)$, even if $\gamma$ itself is not unit-speed.    
\end{rem}

As a consequence of Theorem \ref{teo:evolutoideregularidade}, it's natural to ask what kind of singularities might appear on the evolutoid $\gamma_{\alpha}$. To answer this question, first note that in \cite{Gibson_Hobbs_1993}, C. G. Gibson and C. A. Hobbs classified simple singularities of space curves. In the following theorem, we present a condition for which the curve $\gamma_{\alpha}$ has a singularity of type $(2,3,0)$. This means that $\gamma_{\alpha}$ is right-left equivalent to the normal form $(s^2, s^3, 0)$ at its singularity. For definitions and more details on right-left equivalency, readers can consult any singularity theory textbook, for instance \cite{bruce1992curves,FaridBook}.

\begin{teo}\label{teo:CuspSingularity}
Let $c=1,-1$. Following the assumptions stated in Theorem \ref{teo:evolutoideregularidade}, suppose that $\gamma_\alpha$ is singular at $s=s_0$. This singularity is of type $(2,3,0)$ if and only if $\rho'_{\alpha}(s_0) = - \cos(\alpha)$, $\rho''_{\alpha}(s_0) \ne 0$, and the following open conditions are met (note that all functions are considered at $s=s_0$, which has been dropped for simplicity):
    \begin{eqnarray*}
       \cos_c(\rho_{\alpha})\Big( k \rho_{\alpha}''\cos_c(\rho_{\alpha}) + b \sin_c(\rho_{\alpha}) \left(a^2+cb^2\rho''_{\alpha}\right) \Big) & \ne & 0,\hspace{1cm}\text{or}\\
       c \sin_c(2\rho_{\alpha}) \big( b\rho'''_{\alpha} + 2a k \rho''_{\alpha}\big)+ 2ab \rho''_{\alpha} \big( c+\sin_c^2(\rho_{\alpha})\big) &\ne & 0,\hspace{1cm}\text{or}\\       
       c \Big( 2b^2\sin_c^2(\rho_{\alpha}) - a^2 \cos_c^2(\rho_{\alpha})\Big) + b k \sin_c(2\rho_{\alpha}) \rho''_{\alpha} & \ne & 0.           
    \end{eqnarray*}
\end{teo}

\begin{proof}
Due to Theorem \ref{teo:evolutoideregularidade}, let the $\alpha$-evolutoid $\gamma_{\alpha}$ be singular at $s=s_0$, i.e., $\rho'_{\alpha} = -a$ at $s=s_0$. The curve $\gamma_{\alpha}$ is right-left equivalent to the normal form $(s^2,s^3,0)$ at $s=s_0$ if the rank of the matrix $[\gamma_{\alpha}''(s_0) , \gamma'''_{\alpha}(s_0)]$ is $2$. This means that $\gamma_{\alpha}''(s_0) \wedge_c \gamma_{\alpha}'''(s_0) \ne 0$. The rest of the proof is straightforward but lengthy, so we omit it.
\end{proof}

Differentiating (\ref{eq:raiodecurvaturainclinado2}) and using Theorem \ref{teo:evolutoideregularidade}, we have that $s$ is a singularity of $\gamma_\alpha$ when $G(s,\alpha)=0$, where $G : [0,\ell] \times [0,\frac{\pi}{2}] \longrightarrow \R$ is given by
$$G(s,\alpha) = \rho_\alpha'(s) + \cos(\alpha) = \frac{-\sin(\alpha) k'(s)}{k(s)^2+c \sin(\alpha)^2}+\cos(\alpha).$$
Consider the set
$$\begin{array}{ccl} 
A & = & \{\alpha \in [0,\frac{\pi}{2}] \, ; \, \gamma_\alpha \,\, \text{is singular}\}\\
& = & \{\alpha \in [0,\frac{\pi}{2}] \, ; \exists s \in [0,\ell] \,\, \text{with} \,\, G(s,\alpha)=0\}.
\end{array}$$
Since $\frac{\pi}{2} \in A$, we have that $A \neq \varnothing$. Define $\alpha_0 = \inf A$. Note that $G^{-1}(0)$ is compact and $A$ is the projection of $G^{-1}(0)$ onto the second coordinate. Hence, $A$ is compact and $\alpha_0 \in A$. Therefore, $\gamma_{\alpha_0}$ is singular.
\begin{teo}
    The singularities of $\gamma_{\alpha_0}$ are not of type $(2,3,0)$.
\end{teo}
\begin{proof}
    Suppose that $s=s_0$ is a singularity of $\gamma_{\alpha_0}$, that is, $G(s_0,\alpha_0) = 0$. Note that $G_s (s_0,\alpha_0) = 0$ because $\alpha_0 = \inf A$ (see Figure \ref{fig:alpha0}). Thus
    $$\rho_{\alpha_0}''(s_0) = G_s(s_0,\alpha_0) = 0.$$
    Therefore, it follows from Theorem \ref{teo:CuspSingularity} that $s_0$ is not a singularity of type $(2,3,0)$.
    \begin{figure}[h!]
	\centering
	\includegraphics[width=0.6\textwidth]{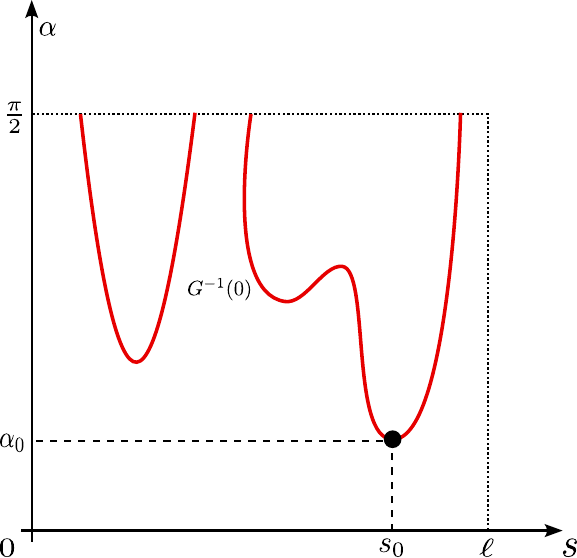}
        \caption{Geometric representation of $\alpha_0$.}
        \label{fig:alpha0}
    \end{figure}
\end{proof}

Now we turn our attention to equation \eqref{eq:parametrizacao-evolutoide}, which implies readily that if $\gamma:[0,\ell]\longrightarrow M_c$ is a circle, then $\gamma_\alpha$ are circles as well, for each $\alpha\in[0,\pi/2)$. Jerónimo-Castro obtained the converse for this claim on the plane case (\cite{jeronimo2014-Aequationes}, Theorem 4), by means of support functions and some properties of constant-width curves. It is natural to ask whether this is the case in space forms. In this sense, using only \eqref{eq:parametrizacao-evolutoide}, we give a simple proof for the converse claim on space forms. Let us begin by providing the geodesic curvature of $\gamma_\alpha$.

\begin{lem}Let $\gamma:[0,\ell]\longrightarrow M_c$ a closed, smooth, convex curve, and $\alpha\in[0,\pi/2)$ such that $\gamma_\alpha$ is smooth. Let $k=k(s)$ be the geodesic curvature of $\gamma$. Then the geodesic curvature of $\gamma_\alpha$ is given by
	\begin{equation}\label{eq:geodesic-curvature-evolutoid}
		k_\alpha(s)=\frac{(k^2(s)+cb^2)^{3/2}}{a(k^2(s)+cb^2)-bk'(s)}.
	\end{equation}
\end{lem}
\begin{proof}It follows readily from deriving \eqref{eq:regularidadeevolutoides} and using \eqref{geodesic-curvature-formula}.
\end{proof}

\begin{teo}Let $\gamma:[0,\ell]\longrightarrow M_c$ a closed, smooth, convex curve, and $\alpha\in[0,\pi/2)$ such that $\gamma_\alpha$ is smooth. Let $k=k(s)$ be the geodesic curvature of $\gamma$. Then $\gamma$ is a circle if and only if $\gamma_\alpha$ is a circle.
\end{teo}
\begin{proof}We just need to prove the ``only if'' claim. Suppose that $\gamma_\alpha$ has constant geodesic curvature $k_0$. We prove that $k$ is constant. By \eqref{eq:geodesic-curvature-evolutoid},
$$k_0(a(k^2(s)+cb^2)-bk'(s))=(k^2(s)+cb^2)^{3/2}.$$
Let $s_0,s_1\in\R$ be any two critical points of $k$. Evaluating the above equation on $s_0,s_1$ gives
        $$ak_0(k^2(s_0)+cb^2)^{1-3/2}=ak_0(k^2(s_1)+cb^2)^{1-3/2},$$
implying $k(s_0)=k(s_1)$, so $k$ is constant.
\end{proof}

\section{Area and length of the $\alpha$-evolutoid}\label{section:area}

In this section, we discuss some aspects concerning the total areas enclosed by a closed, convex curve $\gamma$ in $M_c$ and its $\alpha$-evolutoid (for small $\alpha$), as well as their lengths.

It is naturally expected that the total area enclosed by $\gamma_\alpha$ is a decreasing function of (small) $\alpha\in[0,\pi/2]$ for a fixed $\gamma$. Let $A_c(\gamma)$ denote the area of a domain $\Gamma\subset M^2_c$ enclosed by the convex, regular curve $\gamma=\partial\Gamma$. If $c=0$ and $\gamma_\alpha$ is a regular $\alpha$-evolutoid of $\gamma$, then Jerónimo-Castro (\cite{jeronimo2014-Aequationes}, Theorem 1) proved that 
\begin{equation}\label{desig:jeronimo}
	\frac{A_0(\gamma_\alpha)}{A_0(\gamma)}\leqslant\cos^2(\alpha),
\end{equation}
with the equality holding if and only if $\gamma$ is any circle. Notice that the quotient in \eqref{desig:jeronimo} doesn't depend on the radius of such circles. But in the hyperbolic case, we have a rather different situation: if $\gamma^R,\gamma^R_\alpha\subset M_{-1}$ are respectively a geodesic circle of radius $R$ and its $\alpha$-evolutoid, then the quotient $A_{-1}(\gamma^R_\alpha)/A_{-1}(\gamma^R)$ depends on $R$. We depict on Figure \ref{quocienteareas} the behavior of
\begin{equation}\label{areaacomparing}
	\mathcal{A}_{c}(R)=A_{c}(\gamma_\alpha^R)/A_{-1}(\gamma^R), R>0
\end{equation}
for $c=-1,0,1$ and observe that, at least for geodesic circles, the hyperbolic counterpart to the \textit{optimal} upper bound in \eqref{desig:jeronimo} is an asymtotic function of $R$. One can naturally relates the value $\mathcal{A}_{-1}(0)=\cos^2(\alpha)$ to the fact that the area of very small domains tend to feel little effect of the Gaussian curvature when compared to each other. It would be of interest to determine how the optimal upper bound for \eqref{areaacomparing} depends (asymptotically) on the \textit{diameter} $R>0$ of a given convex, closed curve $\gamma\subset M_{-1}$ (and perhaps on other geometric features of $\gamma$). We expect that in each diameter class of curves (i.e., two convex curves are equivalent if and only if they have the same diameter) the equality in the hyperbolic version of \eqref{desig:jeronimo} holds if and only if the curve is a geodesic circle in this class, although we don't have yet a proof of this result.

    \begin{figure}[h!]
	\centering
	\includegraphics[width=0.85\textwidth]{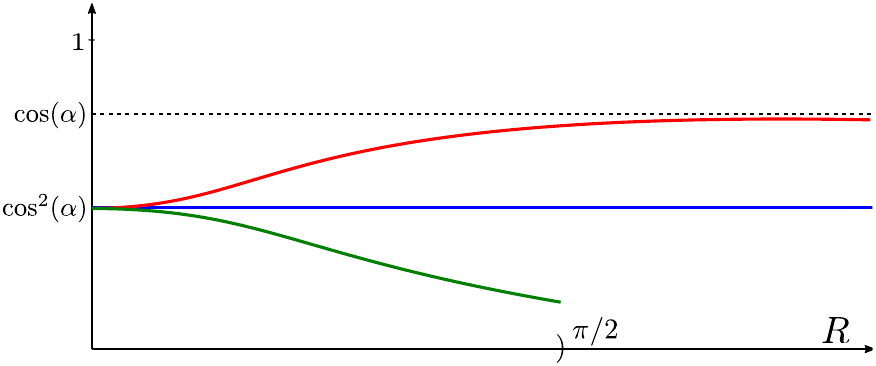}
        \caption{The behaviour of $A_{c}(\gamma_\alpha^R)/A_{c}(\gamma^R)$ with $\gamma^R$ a geodesic circle of radius $R>0$, for $c=-1$ (red), $c=0$ (blue) and $c=1$ (green).}
        \label{quocienteareas}
    \end{figure}

On the other hand, for $c=-1,1$, we extend Jerónimo-Castro's Lemma 2 from \cite{jeronimo2014-Aequationes}. More specifically, given a convex, closed curve $\gamma:[0,\ell]\longrightarrow M_{-1,1}$, we calculate the length of a family of its smooths $\alpha$-evolutoids $\gamma_\alpha$. Before proceeding, we highlight a technical aspect of equation \eqref{eq:evolutoideregularidade}: since  $\gamma_0=\gamma$, hence $\rho_0(s)=\rho_0'(s)=0$ for all $s\in[0,\ell]$. So the condition \eqref{eq:evolutoideregularidade} is readily fulfilled:
$$0=\rho_0'(s)>-\cos(0)=-1.$$
By continuity, there exists a (maximal) $\alpha^*\in(0,\pi/2)$ such that
\begin{equation}\label{eq:positivitycondition}
	|\rho_\alpha'(s)|<\cos(\alpha),
\end{equation}
for all $\alpha\in[0,\alpha^*)$ and for all $s\in[0,\ell]$. We shall use this condition in Proposition \ref{teo:comprimeinto-evolutoide}.

\begin{prop}\label{teo:comprimeinto-evolutoide}Let $\gamma:[0,\ell]\longrightarrow M_c$ a convex, closed curve, $\gamma_\alpha$ its $\alpha$-evolutoid and denote by $L(\gamma),L(\gamma_\alpha)$ their respective lengths. If $\alpha\in[0,\alpha^*)$, then
\begin{equation}\label{eq:comprimento-evolutoide}
	L(\gamma_\alpha)= \cos(\alpha)\cdot L(\gamma)
\end{equation}
\end{prop}
\begin{proof} We can assume $\|\gamma'\|_c=1$, i.e.\ $\ell=L(\gamma)$. The result follows readily from \eqref{derivadaevolutoide} and the condition \eqref{eq:positivitycondition}. We have
$$\|\gamma_\alpha'\|_c=|\rho_\alpha'+a|=\rho_\alpha'+a.$$
Integrating in $[0,\ell]$ gives us \eqref{eq:comprimento-evolutoide}. 
\end{proof}


\section{Involutoids in space forms}\label{secaoinvolutoides}

It is natural to ask wether a given closed, piecewise regular curve  $\gamma:[0,\ell]\longrightarrow M_c$ is the $\alpha$-evolutoid of some curve $\eta$. When this is the case, we still have the problem of \textit{uniqueness} of $\eta$. On the plane case, consider a circle $\gamma$ (Figure \ref{fig:involutoidcirculo}). It has a family of $\alpha$-involutoids, most of them being exponential spirals, and a closed circle. In $\R^2$, this is the general framework: if $\gamma$ is a closed, smooth, convex plane curve, there is a unique closed involutoid $\eta$ for $\gamma$ in $\R^2$ (\cite{jeronimo2014-Aequationes}, Lemma 3). Our aim is to extend this result to the hyperbolic space using Massera's Theorem \cite{massera1950existence}.

\begin{definition}\label{defi:involutoide}An \emph{$\alpha$-involutoid} of a given closed, piecewise regular, arc-length parametrized curve $\gamma:[0,\ell]\longrightarrow M_c$ is any regular curve $\eta:I\subset\R\longrightarrow M_c$ satisfying \begin{equation}\label{eq:involutoide1}
		\eta_\alpha=\gamma.
	\end{equation}
We say that $\eta$ is an $\alpha$-involutoid associated to $\gamma$.
\end{definition}
\begin{figure}[h!]
	\centering
    \includegraphics[width=0.95\textwidth]{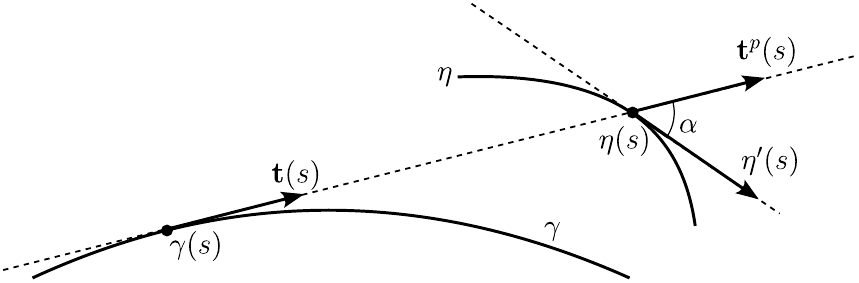}
    \caption{$\eta$ is an $\alpha$-involutoid of the curve $\gamma$.}
    \label{fig:involutoid}
\end{figure}
We explore conditions for finding any $\alpha$-involutoid associated to a given closed, smooth, convex curve $\gamma \in M_c$. Fix $\alpha\in(0,\pi/2)$. If $\eta$ is an $\alpha$-involutoid associated $\gamma$, then \eqref{eq:exponentialmaps} implies that for each $s$ there exists $\lambda(s)\in\R$ satisfying

\begin{equation}\label{eq:involutoidcondition1}
	\eta(s)=\cos_c(\lambda(s))\gamma(s)+\sin_c(\lambda(s))\vec{t}(s)
\end{equation}

\noindent and
\begin{equation}\label{eq:involutoidcondition2}
	\cos(\alpha)=\dfrac{\left\langle \eta'(s),\vec{t}^p(s)\right\rangle_c}{\|\eta'(s)\|_c\|\vec{t}^p(s)\|_c},
\end{equation}
where $\vec{t}^p(s)$ is the parallel transport of $\vec{t}(s)$ along the tangent geodesic to $\gamma$ at $\gamma(s)$ (see Figure \ref{fig:involutoid}), i.e.
\begin{equation}
\vec{t}^p(s)=\frac{d}{dt}\Big|_{\lambda(s)}\exp^c_{\gamma(s)}(t\gamma'(s)).
\end{equation}
After straightforward calculations, we obtain from equation \eqref{eq:involutoidcondition2} that
\begin{equation}\label{eq:involutoidcondition3}
\dfrac{1+\lambda'(s)}{\sqrt{(1+\lambda'(s))^2+(k(s))^2\sin_c^2(\lambda(s))}}=\cos(\alpha)
\end{equation}
where $k(s)$ is the geodesic curvature of $\gamma$ at $\gamma(s).$ This equation is equivalent to
\begin{equation}
    \lambda'(s)=\cot(\alpha)k(s)|\sin_c(\lambda(s))|-1.\label{edoinvolutoides}
\end{equation}
\noindent Therefore all the $\alpha$-involutoids are given by the family of solutions $\lambda=\lambda(s)$ to \eqref{edoinvolutoides}. Figure \ref{fig:involutoidcirculo} illustrates some $\alpha$-involutoids of a plane circle.
\begin{figure}[h]
	\centering
    \includegraphics[width=0.75\textwidth]{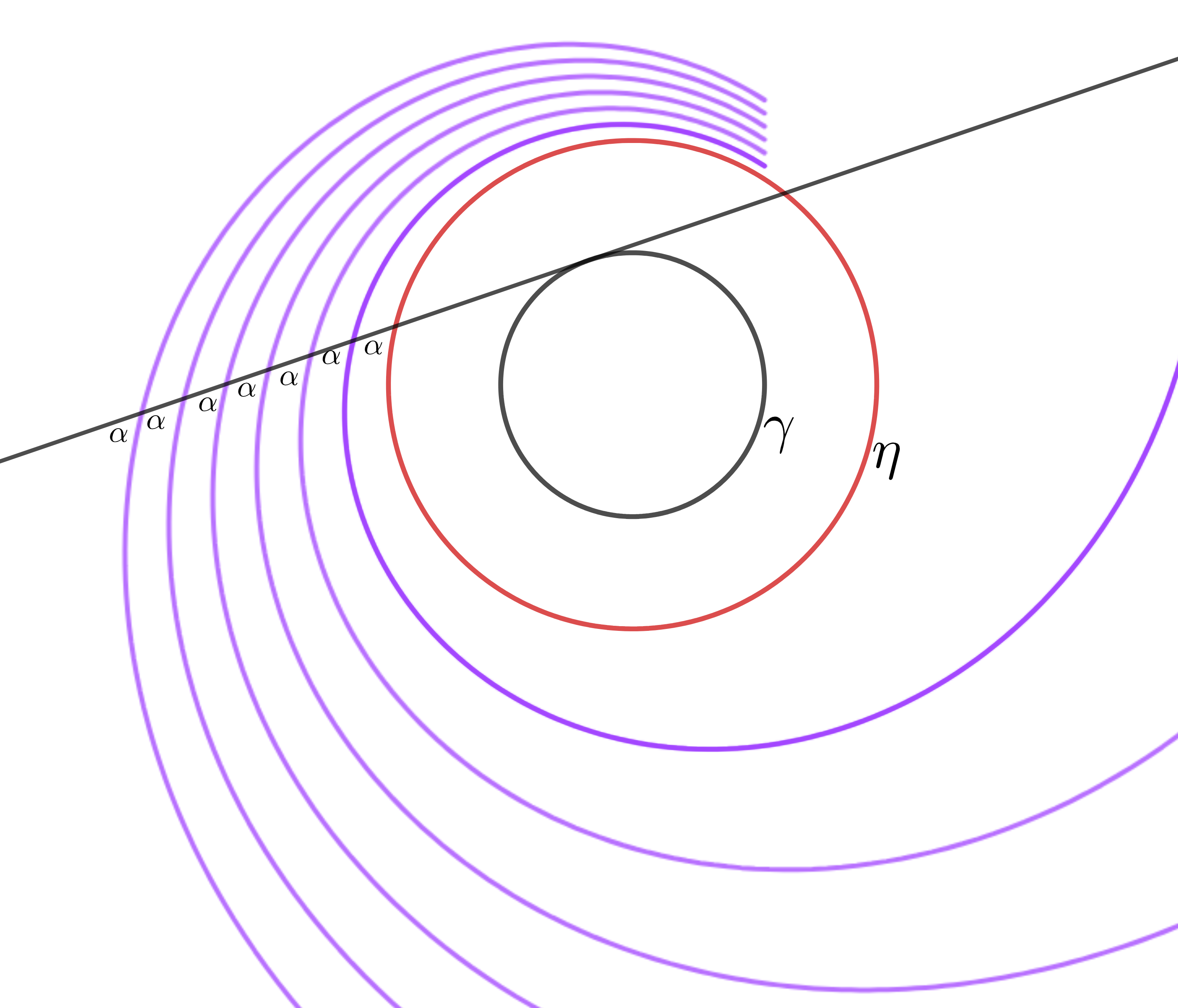}
    \caption{Family of involutoids of the euclidean circle $\gamma$, with $\eta$ being a closed one.}
    \label{fig:involutoidcirculo}
\end{figure}

\begin{prop}\label{sing_involutoid} Let $\gamma:[0,l]\longrightarrow M_c$ be a regular curve and $\eta$ an $\alpha$-involutoid associated to $\gamma$. The singular points of $\eta$ belong to $\gamma$.    
\end{prop}

\begin{proof}It is a straightforward calculation.
\end{proof}

We claim that \eqref{edoinvolutoides} always has a periodic solution if $c=-1,0$, which means any plane or hyperbolic closed convex curve has an associated closed $\alpha$-involutoid. We use the well-known

    \begin{teo}[Massera, J. \cite{massera1950existence}] Let $F:(-\infty,s_0]\times\mathbb{R}\longrightarrow\mathbb{R}$ a locally lipschitz function satisfying $F(s+T,y)=F(s,y)$ for some $T>0$ and for all $(s,y)\in(-\infty,s_0]\times\mathbb{R}$. Then the problem
    \begin{equation}\label{massera}
    y'(s)=F(s,y(s))    
   \end{equation}
   has a periodic solution in $(-\infty,s_0]$ if and only if it has a bounded solution on $(-\infty,s_0]$. In this case, all solutions $\overline{y}$ for the problem are asymptotic to the periodic solution, i.e.\, $|y(s)-\overline{y}(s)|\rightarrow0$ when $s\rightarrow-\infty$.
    \end{teo}

In particular, there exists at most one periodic solution to \eqref{massera}. To see that \eqref{edoinvolutoides} fulfill the hypothesis form Massera's Theorem, we need some geometric results about functions $y$ satisfying
\begin{equation}\label{edogeralinvolutoides}
    y'(s)=f(s)|g(y(s))|-1
\end{equation}
where $f$ is a positive bounded function and $g$ is such that $|g|:\mathbb{R}\longrightarrow\mathbb{R}$ is unbounded in $[0,+\infty)$ with (no loss of generality) $g(0)=0$.

    \begin{lem}\label{lemainvolutoides1}For all $s_0\in\mathbb{R}$, there exists $A<B$ such that any solution $y:(-\infty,s_0]\longrightarrow\mathbb{R}$ to \eqref{edogeralinvolutoides} satisfy:
        \begin{itemize}
            \item[i)] if $y(s)=B$, for some $s< s_0$, then $y'(s)>0$.
            \item[ii)] if $y(s)=A$, for some $s< s_0$, then $y'(s)<0$.
        \end{itemize}
    \end{lem}
    \begin{proof}We prove the first claim. Denote $0<f_\text{min}\leqslant f(s)\leqslant f_\text{max}$ for all $s\in(-\infty,s_0]$. Since $|g|$ is unbounded on $[0,+\infty)$, there exists $y=B$ such that $f_\text{min}|g(B)|-1>0$. Now let $y:\mathbb{R}\longrightarrow\mathbb{R}$ be a solution to \eqref{edogeralinvolutoides} with $y(s)=B$ for some $s< s_0$. By construction, we have $y'(s)=f(s)|g(B)|-1>0$. The second claim can be shown analogously, recalling that $g(0)=0$ and hence there exists $y=A$ such that $f_\text{max}|g(A)|-1<0$.
    \end{proof}

Of course the argument above works under different hypothesis on $g$ and combinations of signs and constants on \eqref{edogeralinvolutoides}, provided that $g$ has large enough range for controlling the sign of $\lambda'$. Despite that, we have:
    
    \begin{coro}Fix $s_0\in\mathbb{R}$. Under the same hypothesis from Lemma \ref{lemainvolutoides1}, the equation \eqref{edogeralinvolutoides} has a bounded solution defined on $(-\infty,s_0]$.
    \end{coro}
    \begin{proof}In fact, we claim that any solution $y$ with $y(s_0)\in[A,B]$ given by Lemma \ref{lemainvolutoides1} must be within $[A,B]$ on $(-\infty,s_0]$. If this was not the case, there would be a solution $y:(-\infty,s_0]\longrightarrow\mathbb{R}$ of \eqref{edogeralinvolutoides} and $s<s_0$ such that $y(s)>A$ (resp.\ $y(s)<B$). By compactness, there exists $s<\overline{s}\leqslant s_0$ such that $y(\overline{s})=A$ (resp.\ $=B$) but with $y'(\overline{s})\leqslant 0$ (resp.\ $\geqslant 0$), a contradiction.
    \end{proof}
    
Specializing the discussion to our equation \eqref{edoinvolutoides} and using Massera's Theorem, we have:

    \begin{teo}Any closed, convex curve $\eta:\mathbb{R}\longrightarrow M_{-1,0}$ has an unique closed $\alpha$-involutoid associated, for each fixed $\alpha\in(0,\pi/2)$.
    \end{teo}

As a result, all non-closed involutoids of convex, closed  curves in $M_{-1,0}$ are unbounded (for $s\rightarrow+\infty$) and asymptotic (for $s\rightarrow-\infty$) to the original curve. The eventually negative solutions to \eqref{edoinvolutoides} give rise to singular involutoids, whose singularities occur exactly when $\lambda(s)=0$ and are related to the so-called \textit{wavefronts} (see Theorem \ref{involutoidesewavefronts}).

Recall that convexity in $M_{-1}$ means $k(s)>1$ for all $s$, but \eqref{edoinvolutoides} has solutions demanding just positivity and boundedness of $f$. Hence any closed curve in the plane or in the hyperbolic space with \textit{strictly positive curvature} has an associated smooth closed $\alpha$-involutoid, $\alpha\in(0,\pi/2)$.

It is interesting to observe what happens on the spherical case $M_{1}$. Equation \eqref{edoinvolutoides} becomes
$$\lambda'(s)=\cot(\alpha)k(s)|\sin(\lambda(s))|-1.$$
By the boundedness of $|\sin(\cdot)|$, we may not have enough room for controlling the sign of $\lambda'$ (as we have done in the proof of Lemma \ref{lemainvolutoides1}) depending on specific characteristics of $k$ (i.e.\ of $\eta$). As a consequence, we may not guarantee Massera's conditions, as we show in the next example.

%
\begin{exa} If $c=1$, the non constant solutions $\lambda$ to \eqref{edoinvolutoides} are:
        \begin{itemize}
            \item[a)] If $\sin(\lambda)>0$ then
                \begin{equation}
                s+K =
                    \begin{cases}
                    -\dfrac{1}{B}\ln\left|\dfrac{\tan\left(\tfrac{\lambda}{2}\right)-A-B}{\tan\left(\tfrac{\lambda}{2}\right)-A+B}\right|, & \text{if } A>1, \\[1.2em]
                    \dfrac{2}{\tan\!\left(\tfrac{\lambda}{2}\right)-1}, & \text{if } A=1, \\[1.2em]
                    -\dfrac{2}{B}\arctan\!\left(\dfrac{\tan\!\left(\tfrac{\lambda}{2}\right)-A}{B}\right), & \text{if } A<1.
                    \end{cases}
                \end{equation}
            \item[b)] If $\sin(\lambda)<0$ then  
                \begin{equation}
                s+K =
                    \begin{cases}
                    -\dfrac{1}{B}\ln\left|\dfrac{\tan\!\left(\tfrac{\lambda}{2}\right)+A-B}{\tan\!\left(\tfrac{\lambda}{2}\right)+A+B}\right|, & \text{if } A>1, \\[1.2em]
                    \dfrac{2}{\tan\!\left(\tfrac{\lambda}{2}\right)+1}, & \text{if } A=1, \\[1.2em]
                    -\dfrac{2}{B}\arctan\!\left(\dfrac{\tan\!\left(\tfrac{\lambda}{2}\right)+A}{B}\right), & \text{if } A<1.
                    \end{cases}
                \end{equation} 
            \noindent where $K$ is constant, $A=\cot(\alpha) k_0$ and $B=\sqrt{A^2-1}$.
        \end{itemize}
For constant solutions $\lambda(s)=\lambda_0$ \eqref{edoinvolutoides} reduces to
$$\sin(\lambda_0)=\pm\dfrac{\tan(\alpha)}{k_0}$$
Notice that they exist only when $k_0\geq\tan(\alpha)$ and are given by
\begin{itemize}
    \item[a)] If $\sin(\lambda_0)=\dfrac{\tan(\alpha)}{k_0}$
            $$\lambda_0=\operatorname{arcsin}\left(\dfrac{\tan \alpha}{k_0}\right) + 2n\pi$$
            or
            $$\lambda_0=\pi - \operatorname{arcsin}\left(\dfrac{\tan \alpha}{k_0}\right) + 2n\pi$$
            for $n\in\mathbb{Z.}$
    \item[b)] If $\sin(\lambda_0)=-\dfrac{\tan(\alpha)}{k_0}$
            $$\lambda_0=\pi + \operatorname{arcsin}\left(\dfrac{\tan \alpha}{k_0}\right) + 2n\pi$$
            or
            $$\lambda_0=2\pi - \operatorname{arcsin}\left(\dfrac{\tan \alpha}{k_0}\right) + 2n\pi$$
            for $n\in\mathbb{Z.}$
\end{itemize}
In all cases, it is not difficult to prove that the only periodic solutions occur when $\lambda$ is constant. The inequality $k_0\geq\tan(\alpha)$ illustrates our choices of $A$ and $B$ on Lemma \ref{lemainvolutoides1} in order to guarantee the boundedness condition on Massera's theorem.
\end{exa}

\section{Evolutoids and wavefronts}\label{wavefronts}

Let $M$ be an oriented $1$-manifold and $(N^2,g)$ an oriented Riemannian $2$-manifold. The unit cotangent bundle $T^*_1N^2$ has the canonical contact structure and can be identified with the unit tangent bundle $T_1N^2$. A smooth map $f: M \longrightarrow N^2$ is called a \textit{front} or \textit{wavefront} if there exists a unit vector field $\nu$ of $N^2$ along $f$ such that $$L=(f,\nu): M \longrightarrow T_1N^3$$
is a Legendrian immersion, that is, the pull-back of the canonical contact form of $T_1N^2$ vanishes on $M$. This condition is equivalent to the following orthogonality condition:
\begin{equation}\label{ortogonalityconditionwavefront}
    g(df(X),\nu)=0, \ \ X\in TM,
\end{equation}
where $df$ is the differential map of $f$. See for instance \cite{saji2009geometry, arnold1994topological, fukunaga2014evolutes} for details. Let $\gamma: I \longrightarrow M_c$ be a smooth, convex, closed curve. Hence, a wavefront of $\gamma$ in de direction $\vec{v}_\alpha(s)=\cos(\alpha) \mathbf{t}(s)+\sin(\alpha) \vec{e}(s)$ is a curve $\sigma: I \longrightarrow M_c$ orthogonal to the geodesic from $\gamma(s)$ in the direction $\vec{v}_\alpha$. More explicitly
\begin{equation}\label{wavefront}
\sigma(s)= \exp^c_{\gamma(s)}(r(s) \vec{v}_\alpha(s)),
\end{equation}
for some function $r(s)$ that satisfies the orthogonality condition, i.e.,
$$ \langle \sigma'(s), \vec{v}^p_\alpha(s)\rangle_c = 0,$$ where $\vec{v}^p_\alpha(s)$ is the parallel transport of the field $\vec{v}_\alpha(s)$ along the slanted geodesic:

\begin{eqnarray}\label{orthogonality}
\vec{v}^p_\alpha(s)&=& \frac{d}{dt}\Big|_{r(s)}\exp^c_{\gamma(s)}(t\vec{v}_\alpha(s))\nonumber\\
&=& -c\sin_c(r(s))\gamma(s)+\cos_c(r(s))\vec{v}_\alpha(s)
\end{eqnarray}
\noindent Using Proposition \ref{teo:unifiedfrenetframes}, it follows from \eqref{orthogonality} that $r'(s)+\cos(\alpha)=0$, and hence
\begin{equation}\nonumber
r(s)= C_0-s\cos(\alpha), \ \ \ C_0\in\mathbb{R}.
\end{equation}
We have proved the following result:

\begin{lem}
The wavefront of a smooth, arc-lenght parametrized, convex curve $\gamma:I\longrightarrow M_c$ in direction $\vec{v}_\alpha(s)$ is given by

\begin{equation}\label{wavefront2}
\sigma(s)= \exp^c_{\gamma(s)}(r(s) \vec{v}_\alpha(s))
\end{equation}
where $r(s)=C_0-s\cos(\alpha).$
\end{lem}

\begin{figure}[H]
    \centering
    \begin{subfigure}[b]{0.25\textwidth}
        \includegraphics[width=\textwidth]{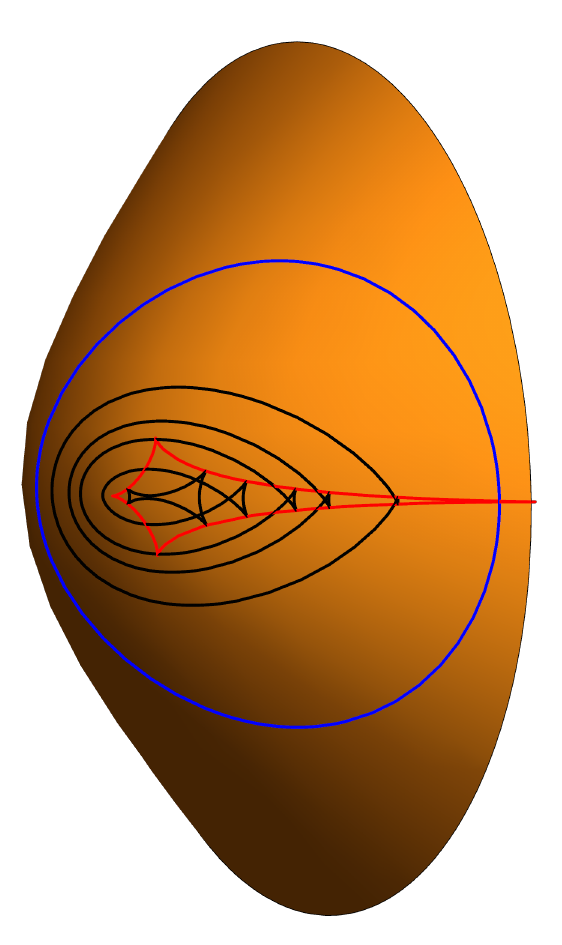}
        \caption{$c=-1$}
    \end{subfigure}
    \quad
    \begin{subfigure}[b]{0.3\textwidth}
        \includegraphics[width=\textwidth]{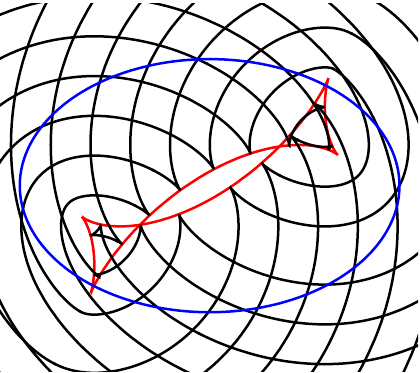}
        \caption{$c=0$}
    \end{subfigure}
    \quad
    \begin{subfigure}[b]{0.35\textwidth}
        \includegraphics[width=\textwidth]{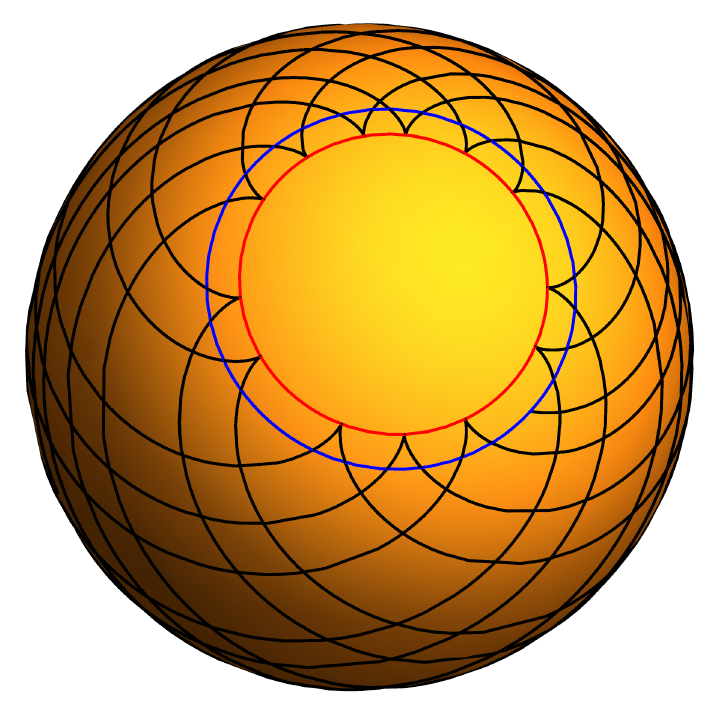}
        \caption{$c=1$}
    \end{subfigure}
    \caption{Examples of wavefronts. The original curve is shown in blue, an $\alpha$-evolutoid in red and the wavefronts in the direction $\vec{v}_{\alpha}$ in black.}
\end{figure}
Now we can present a characterization of the $\alpha$-evolutoids.

\begin{teo}\label{involutoidesewavefronts} Let $\gamma: I \longrightarrow M_c$ a convex closed curve. The $\alpha$-evolutoid of $\gamma$ is the locus of the singular points of the wavefront of $\gamma$ in direction $\vec{v}_\alpha$.
\end{teo}
\begin{proof}
Notice that $\sigma'(s)=\sin(\alpha) A(s)\vec{t}(s)-\cos(\alpha)A(s)\vec{e}(s),$
where $A(s)=\cos_c(r(s)\sin(\alpha)-\sin_c(r(s))k(s)$ and $r(s)=C_0-s\cos(\alpha).$ Thus
$$\|\sigma'(s)\|_c=0 \Longrightarrow C_0-\cos(\alpha)=\mathrm{arcot}_c\left(\dfrac{k(s)}{\sin(\alpha)}\right)$$
and we have $r(s)=\rho_\alpha(s).$
\end{proof}
\vspace{0.5cm}
\subsection*{Acknowledgments} The authors were partially supported by FAPEMIG with process number APQ-02201-22 and APQ-02962-24.

\bibliographystyle{plain} 
\bibliography{Artigo_Evo_2024} 

\end{document}